\documentclass[12pt]{amsart}

\usepackage{amsfonts}
\usepackage{amscd}
\usepackage{amssymb}

\allowdisplaybreaks

\usepackage{enumerate,enumitem,hyperref,cleveref, xcolor}

\date{\today}

\newtheorem{thm}{Theorem}[section]
\newtheorem{cor}[thm]{Corollary}
\newtheorem{lem}[thm]{Lemma}

\newtheorem{prop}[thm]{Proposition}
\theoremstyle{definition}

\theoremstyle{remark}
\newtheorem{rem}[thm]{Remark}

\numberwithin{equation}{section}

\newcommand{\R}{\mathbb R}

\newcommand{\C}{{\mathbb C}}

\newcommand{\Dk}{\Delta_{\kappa}}

\newcommand{\K}{\kappa }

\newcommand{\Fk}{\mathcal{F}_\kappa}

\usepackage{color}

\title[Dunkl multipliers]
{On the boundedness of Dunkl multipliers}

\author[ Mukherjee-Thangavelu]{ Suman Mukherjee and Sundaram Thangavelu}

\address[S. Mukherjee]{School of Mathematical Sciences, National Institute of Science Education and Research, An OCC of Homi Bhabha National Institute, Bhubaneswar--752050, India.}
\email{sumanmukherjee822@gmail.com}

\address[S. Thangavelu]{Department of Mathematics, Indian Institute of Science, Bangalore--560012, India.}
\email{veluma@iisc.ac.in}


\begin{document}

	\maketitle
	
	\vskip0.25in

	\begin{abstract} In this article we use Littlewood-Paley-Stein theory to  prove two versions of Dunkl multiplier theorem when the multiplier $ m $ satisfies a modified H\"ormander condition. When $ m $ is radial we give a simple proof of a known result. For general $ m $ we prove that the Dunkl multiplier operator takes radial functions in $ L^p $ boundedly into $ L^p $ for all $ p \geq 2.$
	\end{abstract}

	
	\section{Introduction} \label{Sec-intro} In this work we are concerned with a multiplier theorem for the Dunkl transform  $ \Fk$ which is a generalisation of the Fourier transform $ \mathcal{F}.$
	We begin by recalling some of the  known results on Fourier multipliers.  By a Fourier multiplier operator we mean an operator of the form $ T_m $ defined by the relation $ \mathcal{F}(T_mf )= m\, \mathcal{F}f $ for $ f \in L^2(\R^n)$ where $ m \in L^\infty(\R^n) $ is known as the multiplier associated to $ T_m.$ These are precisely the operators on $ L^2(\R^n) $ which commute with all translations.  It is well known that unless some extra conditions are imposed on $m $ the operator $ T_m $ need not extend to $L^p(\R^n) $ as a bounded operator for $ p \neq 2.$ The celebrated H\"ormander-Mihlin multiplier theorem gives a sufficient condition on $ m $ so that $ T_m $ is bounded on $ L^p(\R^n) $ for all $ 1 < p < \infty.$\\
	
	Given  $ m \in L^\infty(\R^n) ,  \psi \in C_0^\infty(\R^n) $ supported in the complement of $ 0$ and  $t >0$, we define
	$$  \widehat{m}_\psi(x,t) = \int_{\R^n} e^{i x \cdot \xi} \psi(\xi) m( \xi/\sqrt{t})\, d\xi.$$ 
	We say that $ m $ satisfies  H\"ormander's condition of order $ s \geq 0 $ if  for any such $ \psi$
	\begin{equation}\label{HM} \int_{\R^n} (1+|x|^2)^s | \widehat{m}_\psi(x, t)|^2 dx \leq C_\psi 
	\end{equation} 
	for all $ t >0.$ Then the following theorem is well known in the literature.
	
	\begin{thm} \label{loc-sob}Let $ m $ be a bounded function which satisfies the condition \eqref{HM} for some $ s> n/2.$ Then the Fourier multiplier $ T_m $ is bounded on $ L^p(\R^n) $ for any $ 1 < p < \infty.$
	\end{thm}
	
	An important corollary of this theorem is the following version of H\"ormander-Mihlin multiplier theorem with easy to verify conditions on $m.$
	
	\begin{cor} Let $ m $ be a bounded function which satisfies  $ |\partial^\alpha m(\xi)| \leq C_\alpha |\xi|^{-|\alpha|} $ for all  $ |\alpha| \leq k $ for some $ k > n/2.$ Then  $ T_m $ is bounded on $ L^p(\R^n) $ for any $ 1 < p < \infty.$
	\end{cor}
	
	Dunkl multipliers are defined in the same way as the Fourier multipliers simply by replacing $ \mathcal{F} $ by $ \Fk.$ Thus for $ m \in L^\infty(\R^d) $ we define $ \Fk(\mathcal{T}_m f) = m\, \Fk(f).$ Just like the Fourier transform, the Dunkl transform is unitary on $ L^2(\R^d, h_\K^2) $ for a suitable weight function $ h_\K $ which is homogeneous of degree $ \gamma_\K \geq 0.$ An interesting open problem is to find optimal conditions on $ m $ so that $ \mathcal{T}_m $ initially defined on $ L^2 \cap L^p(\R^d, h_\K^2) $ extends as a bounded operator on $L^p(\R^d, h_\K^2) $ for all $ 1 < p < \infty.$  This problem  has bee studied by  several authors, see  \cite{AGS, ABDM, S} but optimal results are known only when $ d =1.$ In the higher dimensional case with arbitrary root system and non-negative multiplicity function, only partial results are known, notable among them is a recent result of Dziuba\'{n}ski and Hejna \cite{HejnaHMT} which is stated in  the following:
	
	\begin{thm} Let $ m $ be a bounded function which satisfies the condition \eqref{HM} for some $ s> d_\K$  where $ d_\K = d+2\gamma_\K $ is the homogeneous dimension of the measure $ h_\K^2(x) dx.$ Then the Dunkl multiplier $ \mathcal{T}_m $ is bounded on $ L^p(\R^d, h_\K^2) $ for any $ 1 < p < \infty $ and is of weak type $ (1,1).$
	\end{thm}
	
	As in the case of the Fourier transform it is expected that the condition \eqref{HM} for some $ s> d_\K /2$ is sufficient for the boundedness of $ \mathcal{T}_m $ on $ L^p(\R^d, h_\K^2) .$  The main obstacle in reaching this optimal result seems to be the missing information that the Dunkl translation operators $ \tau_\K(x) $ are bounded on $ L^p(\R^d, h_\K^2) .$  Indeed, in the above paper  Dziuba\'{n}ski and Hejna have shown that once $ \tau_\K(x) $ is bounded on $ L^1(\R^d, h_\K^2) ,$  the optimal result holds.\\
	
	Though the boundedness of $ \tau_\K(x) $ is not known on  $ L^p(\R^d, h_\K^2) ,$ it has been proved in  \cite{ThangaveluCOMF}, \cite{GorbachevPLBDTGTOAIA}, that $ \tau_\K(x) $  takes all radial functions in $ L^p(\R^d, h_\K^2) $ boundedly into $ L^p(\R^d, h_\K^2) .$ In view of this we can expect a better multiplier theorem  when $ m $ is radial. Indeed, such a version of multiplier theorem has been proved by  Dai and Wang \cite{WangATT}, see also \cite{DeleavalDSMWVIUL}.
	
	\begin{thm}\label{dai-wang}Let $ m(\xi) = m_0(|\xi|)$ where $m_0 $ is  a bounded function which  for any $R>0$ satisfies the H\"ormander condition
		\begin{equation}
			\frac{1}{R} \int_{R}^{2R} \left| m_0^{(j)}(t) \right|\, dt\lesssim R^{-j},  
		\end{equation}
		where $ j $ is the smallest integer greater than or equal to $ d_\K/2+1/2.$
		Then $ \mathcal{T}_m $ is bounded on $ L^p(\R^d, h_\K^2) $ for any $ 1 < p < \infty .$ 
	\end{thm}
	
	We remark that the above condition on $ m_0 $ holds whenever $ m $ satisfies the estimates $ |\partial^\alpha m(\xi)| \leq C_\alpha |\xi|^{-|\alpha|} $ for all  $ |\alpha| \leq k , k  \geq d_\K/2+1/2.$  This theorem is proved in \cite{WangATT} by a transference technique first used by Bonami and Clerc in \cite{BC}. The authors in \cite{WangATT} first established a multiplier theorem for h-harmonic expansions on the unit sphere $ S^{d-1} $ from which they deduced a  multiplier theorem for the Dunkl transform.  One of the main goals of this paper is to give a direct proof of the above theorem. In fact using Littlewood-Paley-Stein theory of $g$-functions we prove an analogue of Theorem \ref{loc-sob} for radial Dunkl  multipliers from which analogue of Theorem \ref{dai-wang} is deduced.\\
	
	We now introduce the modified H\"ormander condition in the Dunkl setting.  Let $ E_\K(x,\xi) $ stand for the  Dunkl kernel used in the definition of the  Dunkl transform.  Given $ 0 < \delta \leq 1,$ we  define
	\begin{equation}
		\widehat{m}_{s,\delta}(x,t) = \int_{\R^d} E_\K(ix, \xi) \, |\xi|^{2s} m(t^{-1/2\delta}\xi) e^{-|\xi|^{2\delta}} \, h^2_\kappa(\xi) d\xi.
	\end{equation}
	Note that we have replaced the cutoff function $ \psi $ in the original definition of H\"ormander condition by $ |\xi|^{2s}\,e^{-|\xi|^{2\delta}}.$ 
	We say that $ m $ satisfies  the modified H\"ormander condition of order $ s \geq 0 $ if the following estimate holds:
	\begin{equation}\label{mod H}
		\sup_{t>0}\,\int_{\R^d} (1+|x|^2)^s |\widehat{m}_{s,\delta}(x,t)|^2\, h^2_\kappa(x) dx \leq C_s.
	\end{equation}
	Using Littlewood-Paley-Stein theory for the semigroup generated by the fractional power $ (-\Dk)^{\delta}$ of the Dunkl-Laplacian we prove the following version of multiplier theorem for the Dunkl transform.
	
	\begin{thm}  Let $ m $ be a radial function which satisfies the condition \eqref{mod H} for some $ s >d_\K/2.$ 
		Then the Dunkl multiplier $ \mathcal{T}_m $ is bounded on $ L^p(\R^d,h_\K^2)$ for all $ 1 < p < \infty.$
	\end{thm}

	The method we use to prove the above theorem is the standard one using Littlewood-Paley-Stein theory  which is an elegant and powerful tool in proving multiplier theorems. This method has been used successfully by several authors, e.g. by Strichartz \cite{RS} for spherical harmonic expansions, by the second author of this article for Hermite  expansions, Weyl transform  and  Fourier transform on the Heisenberg group \cite{ST1, ST2,ST3}. We are also able to prove the following theorem without assuming $ m$ is radial but considering only radial functions in the domain of the multiplier transform.
	
	\begin{thm}  Let $ m $ be a bounded function which satisfies the condition \eqref{mod H} for some $ s >d_\K/2.$ 
		Then the Dunkl multiplier $ \mathcal{T}_m $ is bounded  from  $ L^p_{rad}(\R^d,h_\K^2 )$ into $ L^p(\R^d,h_\K^2) $ for all $ 2 \leq p < \infty.$
	\end{thm}
	
	In the above  $ L^p_{rad}(\R^d,h_\K^2)$ stands for the subspace of radial functions in $ L^p(\R^d,h_\K^2).$ From the above theorem we are  also able to deduce the following corollary.  Let $ \mathcal{D}_j, j=1,2,...,d $ be the Dunkl derivatives and define $ \mathcal{D}^\alpha  =  \mathcal{D}_1^{\alpha_1}\circ \mathcal{D}_2^{\alpha_2} \circ \cdots \circ \mathcal{D}_d^{\alpha_d},$ for $  \alpha \in \mathbb N^d .$  
	
	\begin{cor}\label{cor} Let $ m $ be a bounded function which satisfies  $ |\mathcal{D}^\alpha m(\xi)| \leq C_\alpha\,  |\xi|^{-|\alpha|} $ for all  $ |\alpha| \leq k $  for some $ k > d_\K/2.$ 
		Then the Dunkl multiplier $ \mathcal{T}_m $ is bounded  from  $ L^p_{rad}(\R^d,h_\K^2 )$ into $ L^p(\R^d,h_\K^2) $ for all $ 2 \leq p < \infty.$
	\end{cor}

	We can rewrite the modified H\"ormander  condition \eqref{mod H}  in terms of another function 
	$$ \widetilde{m}_{s,\delta}(x,t) = \int_{\R^d} E_\K(ix, \xi) \, |\xi|^{2s} m(\xi) e^{-t|\xi|^{2\delta}} \, h^2_\kappa(x) d\xi.$$
	In terms of this function the condition \eqref{mod H} is equivalent to the following estimate:
	\begin{equation}
		\int_{\R^d} (1+t^{-1/\delta}|x|^2)^s |\widetilde{m}_{s,\delta}(x,t)|^2\, h^2_\kappa(x) dx \leq C\, t^{-2s/\delta-d_\K/(2\delta)}.
	\end{equation}
	Since $ E_\K(ix,\xi) $ are eigenfunctions of the Dunkl-Laplacian $ \Dk $ with eigenvalues $ -|x|^2 $ the above  estimate is equivalent to
	\begin{equation}\label{mod-mod H}
		\int_{\R^d} |(1-t^{-1/\delta} \Dk)^{s/2}\big( m(\xi)\, |\xi|^{2s} \, e^{-t |\xi|^{2\delta}} \big)|^2\, h^2_\kappa(\xi) d\xi \leq C\, t^{-2s/\delta-d_\K/(2\delta)}.
	\end{equation}
	When $ s =k $ is an integer, the above boils down to estimating the $ L^2 $ norms of functions of the form $ \mathcal{D}^\alpha \big( m(\xi)\, |\xi|^{2s} \, e^{-t |\xi|^{2\delta}} \big)$  for $ |\alpha| \leq k.$  In Section \ref{subsection-Leib formula for Dunkl deriv} we develop a Leibniz formula for higher order Dunkl derivatives, namely for $ \mathcal{D}^\alpha(mg) $ when $ g $ is radial and use them to verify \eqref{mod-mod H} in proving the above corollary.
	
	\begin{rem} To see the above claim  we can split  the integral into two parts. For $ |x| \leq t^{1/(2\delta)} $ we just use the boundedness of $ m.$  When $ |x| \geq  t^{1/(2\delta)} $ note that 
		$ |x|^{2k} $ is finite linear combination of terms of the form $ x^{2\alpha}, |\alpha| = k.$ Hence we are led to estimate the $L^2 $ norms of terms of the form $ x^\alpha \,\widetilde{m}_{s,\delta}(x,t) $ which on the Dunkl transform side reduces to the estimation of the $ L^2 $ norms of $ \mathcal{D}^\alpha \big( m(\xi)\, |\xi|^{2s} \, e^{-t |\xi|^{2\delta}} \big)$  for $ |\alpha| = k.$
	\end{rem}
	\vskip0.05in
	
	When  $ m $ and $ g $ are radial, the Leibniz formula for $ \mathcal{D}^\alpha(mg) $ has a relatively simpler form which allows us to deduce the following theorem for radial multipliers. 
	
	\begin{thm}\label{rad-mult} Let $ m(x) = m_0(|x|) $ be a bounded radial function which satisfies the decay estimates $ |m_0^{(j)}(r)| \leq C_j\, r^{-j} $ for all  $ j \leq k $  for some $ k > d_\K/2.$ 
		Then the operator $ \mathcal{T}_m $ is bounded on $ L^p(\R^d,h_\K^2)$ for all $ 1 < p < \infty.$
	\end{thm}

	We conclude this introduction with a brief outline of various sections. In the three subsections of  Section \ref{sec-preliminary}, we recall all the  important results from Dunkl theory that are needed in this work. In Section \ref{subsection-Leib formula for Dunkl deriv} we prove a Leibniz formula for higher order Dunkl derivatives. This formula, which is new, plays an important role in verifying H\"ormander condition when the usual decay conditions on Dunkl derivatives of the multiplier are assumed. In Section \ref{sec-revisit Fourier multipliers}, we revisit Fourier multipliers and prove using Littlewood-Paley-Stein theory a Fourier multiplier theorem under modified H\"ormander condition. Though the results in this section are not new, we believe this section will set the stage for what we do with the Dunkl multipliers. Finally in Section \ref{sec-multiplier thm for Dunkl trans} we prove all the multiplier theorems for the Dunkl transform stated in this  section.

	\section{Preliminaries from Dunkl analysis}\label{sec-preliminary} In this section we briefly recall some of the relevant results from our point of view for Dunkl theory.  Our main aim is to set up notation and  state the results that will be used in the sequel. The readers who are not familiar with the Dunkl theory are encouraged to read the references \cite{deJeuTDT, RoslerDOTA, RoslerAPRF} for more details.
	
	\subsection{Dunkl operators and transform}  For any vector  $\lambda\in \mathbb{R}^d$ with $|\lambda|^2 =2 $ we define the operator $\sigma_{\lambda}: \mathbb{R}^d \to \mathbb{R}^d$  by 
	$$\sigma_{\lambda}(x)=x- \langle x,\lambda \rangle\,\lambda,$$
	where $\langle \cdot\ ,\cdot\rangle$ is the usual inner product on $\mathbb{R}^d$. This $\sigma_{\lambda}$ is known  as the \emph{reflection} with respect to the hyperplane $\lambda^{\bot }$.  
	A finite subset $R$ of $\mathbb{R}^d\setminus \{0\}$ is called a \emph{root system} if it satisfies the two conditions: $R \cap \mathbb{R}\lambda = \{ \pm \lambda \} $ for all $\lambda \in R$
	and $\sigma _{\lambda} (R) = R $ for all $\lambda \in R$.
	In this article, we will work with a fixed root system $R$ which is normalized, that is, $|\lambda| = \sqrt{2}$ for all $\lambda \in R$.
	The reflections $\{\sigma_\lambda : \lambda \in R\}$ generate a group $G$, which is a subgroup of the group $O(d, \mathbb{R})$ and is termed the \emph{reflection group} (also known as the \emph{Coxeter group}) related to the root system $R$. \\
	
	A complex-valued $G$-invariant function $\K$ defined on $R$ is known as a \emph{multiplicity function}. Throughout this paper, $\K$ will be fixed and  assumed to be  non-negative. We introduce  the $G$-invariant weight function  $h_\K(x) = \prod_{\lambda \in R^+} |\langle x, \lambda \rangle|^{k(\lambda)}$  where $ R^+ $ is an arbitrary but fixed positive subsystem of $ R.$  Note that $ h_\K(x) $  is homogeneous  of degree $ \gamma_\K = \sum_{\lambda \in R^+ } \K(\lambda) .$ We use the measure $ h^2_\K(x) \, dx $ on $ \R^d $ in defining the Dunkl transform.  Note that  this measure is homogeneous of degree $ d_\K = d+2 \gamma_\K. $ \\
	
	In his seminal work \cite{Dunkl} C. F. Dunkl introduced  the \emph{differential-difference operators} $\mathcal{D}_{\xi}$, known as Dunkl derivatives now, which represent $\K$-deformations of the directional derivative operators $\partial_{\xi}$, coinciding with them when $\K=0$. These are defined by 
	\begin{eqnarray*}
		\mathcal{D}_{\xi} f(x):= \partial_{\xi} f(x)+\sum\limits_{\lambda \in R} \frac{\K(\lambda )}{2} \langle\lambda, \xi \rangle \frac{f(x) - f(\sigma _\lambda x)}{\langle \lambda , x \rangle}.
	\end{eqnarray*}
	An important property of this family of operators is that $ \mathcal{D}_\xi \circ \mathcal{D}_\eta = \mathcal{D}_\eta\circ \mathcal{D}_\xi$ for any $ \xi, \eta \in \R^d.$ When $ \xi = e_j,$ the coordinate vectors, we denote $ \mathcal{D}_\xi $ by $ \mathcal{D}_j.$ The Dunkl derivatives are equivariant under the action of the Coxeter group $ G$ in the following sense. Each   $ \sigma \in G $ induces an  action on functions  $ f: \R^d \rightarrow \C $  which is also denoted by the same symbol $ \sigma$; thus  $ \sigma f(x) = f(\sigma x).$
	With this notation, for any $\xi \in \mathbb{R}^d$ and $\sigma \in G$ we have the relation 
	$$\sigma\circ \mathcal{D}_\xi \circ \sigma ^{-1} = \mathcal{D}_{\sigma (\xi)}.$$
	In particular, for any $\lambda \in R$, 
	$\sigma_\lambda \circ \mathcal{D}_\xi \circ \sigma_\lambda = \mathcal{D}_{\sigma_{\lambda} (\xi)}.$
	For any $ y \in \R^d $ the system of differential equations 
	$$ \mathcal{D}_\xi u(x) = \langle y, \xi \rangle \, u(x), u(0) = 1$$
	has a unique solution denoted by $ E_\K(x,y)$  called the \emph{Dunkl kernel}. Clearly, $E_\K(x,y)$ extends the notion of exponential function $e^{\langle x,y\rangle}$, and it can be uniquely extended to a holomorphic function on $\mathbb{C}^d \times \mathbb{C}^d$. It also enjoys a few properties similar to the exponential function. For example, $E_\K(x,y) = E_\K(y,x)$ for any $x, y \in \mathbb{C}^d$, and $E_\K(tx,y) = E_\K(x,ty)$ for any $x, y \in \mathbb{C}^d$ and $t \in \mathbb{C}$. Moreover, for $x, y\in \R^d$, $|E_\K(ix,y)| \leq 1,$ an important property which allows us to define the Dunkl transform.\\
	
	For any $ f \in L^1(\R^d, h_\K^2) $ its Dunkl transform is defined by 
	\begin{equation}\label{dunkl-trans}
		\mathcal{F}_\K f(\xi) = c_\K \, \int_{\R^d}\, f(x)\, E_\K(-ix, \xi)\, h_\K^2(x)\, dx 
	\end{equation} 
	where $ c_\K $ is the Mehta type constant defined by
	$$c_\K^{-1}=\int_{\mathbb{R}^d}e^{-\frac{1}{2}|x|^2}\,h_\K^2(x)\,dx.$$
	For $ f \in L^1 \cap L^2(\R^d, h_\K^2) $ it turns out that the Plancherel formula $ \|\mathcal{F}_\K f\|_2 = \|f\|_2$ holds and hence $ \mathcal{F}_\K $ extends to the whole of $ L^2(\R^d, h_\K^2) $ as a unitary operator. It shares many properties with the classical Fourier transform. For example, we have the inversion formula: under the assumption that $ f, \mathcal{F}_\K f \in L^1(\R^d, h_\K^2) $ we have
	\begin{equation}\label{dunkl-invers}
		f(x)= c_\K \, \int_{\R^d}\, \mathcal{F}_\K f(\xi)\, E_\K(ix, \xi)\, h_\K^2(\xi)\, d\xi 
	\end{equation} 
	For more about Dunkl transform we refer to the article \cite{deJeuTDT} by Marcel de Jeu.\\

	\subsection{Dunkl translation and convolution} The \emph{Dunkl translation operator} $\tau_\K(x)$, for any $x \in \mathbb{R}^d$, is defined by the relation
	$$\mathcal{F}_\K(\tau_\K(x) f)(y) = E_\K(ix,y) \mathcal{F}_\K f(y)$$
	for all $f \in L^2(\mathbb{R}^d, h_\K^2)$. Since $|E_\K(ix,y)| \leq 1$, the above formula defines $\tau_\K(x)$ as a bounded operator on $L^2(\mathbb{R}^d, h_\K^2)$. In the following proposition, we gather some properties of the Dunkl translations for later use.
	
	\begin{prop}\label{properties of Dunkl trans} 
		(i) For all $f\in \mathcal{S}(\mathbb{R}^d )$, the pointwise formula holds: 
		$$\tau_\K(x)f(y)=\int_{\mathbb{R}^d}\,E_\K(ix,\xi)\,E_\K(iy,\xi) \mathcal{F}_\K f(\xi)\, h_\K^2(\xi)\, d\xi.$$ (ii) The symmetry $\tau_\K(x) f(y) = \tau_\K(y) f(x)$ holds for a reasonable function $f$\\
		
		\noindent  
		and (iii) For any $f\in \mathcal{S}(\mathbb{R}^d)$ and any bounded function $g\in L^1(\mathbb{R}^d, h_\K^2)$ we have
		$$\int_{\mathbb{R}^d}\tau_\K(x)f(y)g(y)\, h_\K^2(y)\, dy= \int_{\mathbb{R}^d}f(y)\tau_\K(-x)g(y)h_\K^2(y)\, dy.$$ 
	\end{prop}

	There is no explicit formula for the Dunkl translation operator which makes it difficult to use. However, when $ f $ is a radial function there is a very useful formula for $ \tau_\K(x)f $ due to R\"osler.
	
	\begin{prop} Let $f$ be a radial Schwartz class function such that $f(x)=f_0(|x|).$ Then
		\begin{equation}\label{trans-rad}
			\tau_\K(x) f(y)=\int_{\mathbb{R}^d} \, f_0(\sqrt{|x|^2+|y|^2+2\langle y, \eta \rangle}) \,d\mu_x(\eta),
		\end{equation}
		where  $\mu_x$ is a probability measure supported in the convex hull of the set $\{ \sigma (x) : \sigma \in G \}$. 
	\end{prop}
	
	The formula \eqref{trans-rad} was initially derived by Rösler \cite[Theorem 5.1]{RoslerAPRF}. Thangavelu and Xu \cite[Proposition 3.3]{ThangaveluCOMF} noted that this formula is valid for radial functions $f$ where both $f$ and $\mathcal{F}_\K f$ belong to $L^1(\mathbb{R}^d, h_\K^2)$. Subsequently, Dai and Wang \cite[Lemma 3.4]{WangATT} expanded this result to include all continuous radial functions in $L^2(\mathbb{R}^d, h_\K^2)$.\\

	From the above proposition we can deduce several properties of $\tau_\K(x)f $ when $ f $ is radial. For example,  if $f$ is  bounded and radial, then  $\tau_\K(x)f\geq 0$ whenever $ f \geq 0.$ As shown in \cite{ThangaveluCOMF} this property need not be true for general functions. We record the following result separately as it will be used in the study of $ g^\ast$-functions.
	
	\begin{cor}\label{trans-prod}
		Let $f$ and $g$ be two continuous radial functions in $L^2(\mathbb{R}^d, h_\K^2).$  Then for any  $x, y \in \mathbb{R}^d$
		$$ |\tau_\K(x)(fg)(y)|\leq \big(\tau_\K(x)|f|^2(y)\big)^{1/2}\,\big(\tau_\K(x)|g|^2(y)\big)^{1/2}.$$
	\end{cor}
	\vskip0.20in
	
	The boundedness of the Dunkl translation operators have been studied by several authors. While it has been proved that  $\tau_\K(x)$ is a bounded operator for radial functions in $L^p(\mathbb{R}^d, h_\K^2)$, $1\leq p\leq \infty$ (see \cite{ThangaveluCOMF}, \cite{GorbachevPLBDTGTOAIA}), it remains an open question whether the Dunkl translation is a bounded operator on the entire space $L^p(\mathbb{R}^d, h_\K^2)$ for $p \neq 2$.\\
	
	The Dunkl convolution of two functions $ f, g \in L^1(\R^d,h_\K^2) $ is defined by 
	$$ f \ast_\K g(x) = \int_{\R^d} \, f(y) \tau_\K(-y)g(x)\, h_\K^2(y)\, dy.$$ 
	It is then easily verified that  the Dunkl transform  converts the convolution into products: $ \mathcal{F}_\K (f \ast_\K g) = \mathcal{F}_\K f\, \mathcal{F}_\K g.$  Therefore, we have  an alternate definition for the convolution. 
	$$ f \ast_\K g(x) =  c_\K\, \int_{\R^d} \, E_\K(ix,\xi)\, \mathcal{F}_\K f(\xi)\, \mathcal{F}_\K g(\xi) h_\K^2(\xi)\, d\xi.$$ 
	From this definition it follows immediately that 
	$ \| f \ast_\K g\|_2 \leq \|g\|_1\, \|f\|_2 .$  Inequalities of the form $ \| f \ast_\K g\|_p \leq \|g\|_1\, \|f\|_p $ for other values of $ p $ are not known in general  due to the lack of knowledge about the boundedness of $ \tau_\K(x).$ However, when $ g $ is radial we do have $ \| f \ast_\K g\|_p \leq \|g\|_1\, \|f\|_p $  as proved in \cite{GorbachevPLBDTGTOAIA, ThangaveluCOMF}.\\
	
	\subsection{Dunkl Laplacian and the heat semigroup} The operator $ \Dk = \sum_{j=1}^d \mathcal{D}_j^2 $ which is the counterpart of the standard Laplacian for the Dunkl setting is known as the Dunkl Laplacian. It is a self-adjoint  non-positive operator on $ L^2(\R^d, h_\K^2) $ and generates a diffusion semigroup $ \{T_t \}_{t\geq 0}= \{e^{-t (-\Dk)}\}_{t\geq 0} $  which is given a by a positive kernel $ q_t .$  Thus $ T_t f = f \ast_\K q_t $ and the kernel is explicitly given by
	\begin{equation}
		q_t(x)= c_\K^{-1}\,  (2t)^{-d_\K/2} e^{-|x|^2/4t}.
	\end{equation}
	It  can be easily verified that $\mathcal{F}_\K q_t(\xi)=c_\K^{-1} e^{-t|\xi|^2}$ and the translated kernel is given by
	$$\tau_\K(-y)q_t(x)= c_\K^{-1}\,  (2t)^{-d_\K/2} e^{-(|x|^2+|y|^2)/4t} E_\K\Big(\frac{x}{\sqrt{2t}}, \frac{y}{\sqrt{2t}}\Big).$$ 
	We denote this kernel by $ h_t(x,y) $ and the semigroup $ \{T_t \}_{t\geq 0} $  can also be defined by
	$$ T_tf(x) = \int_{\R^d} h_t(x,y)\, f(y)\, h_\K^2(y)\, dy.$$
	We record the properties of this semigroup in the following theorem. As the function $ q_t $ is radial  we can use the boundedness of the convolution with $ q_t $ in proving the following result.
	\begin{thm} (See \cite{Roslerold})
		For any $ f \in L^p(\R^d,h_\K^2), 1 \leq p < \infty,$ the function $ u(x,t) = T_tf(x) = f \ast q_t(x) $ solves the heat equation
		$\partial_t u(x, t)=\Dk u(x, t),\,\, u(x, 0)=f(x).$ Moreover, for any $1\leq p \leq \infty$ we have $ \| f \ast_\K q_t\|_p \leq \|f\|_p.$ Furthermore, $ T_tf \geq 0$ whenever $ f \geq 0,$ and $ T_t 1=1.$
		
	\end{thm}

	In order to study the $g_s^\ast$-functions for non-integral values of  $ s $ we need to make use of the semigroup generated by fractional powers of $ (-\Dk).$ For any $ \delta > 0 $ we define
	\begin{equation}
		T_{t,\delta} f(x) = c_\K\, \int_{\R^d} \, \mathcal{F}_\K f(\xi)\, e^{-t|\xi|^{2\delta}}\, E_\K(ix,\xi)\, h_\K^2(\xi)\, d\xi
	\end{equation}
	for $ f \in L^2(\R^d, h_\K^2).$
	For $ 0 < \delta < 1$ we have  Bochner's subordination  formula which allows us  express $ T_{t,\delta} $ in terms of $ T_t$. Thus we have
	\begin{equation}\label{boch-sub} T_{t,\delta} f(x) = \int_0^\infty  T_sf(x)\, \eta_t(s) ds 
	\end{equation}
	where $ \eta_t(s) $ is a probability density function. In view of this, it follows that all the properties of  $T_t $ are shared by $ T_{t,\delta}.$ Hence $ \{T_{t,\delta}\}_{t\geq 0} $ is also a symmetric diffusion semigroup.\\
	
	\section{Leibniz formula for the Dunkl derivatives}\label{subsection-Leib formula for Dunkl deriv} In this section we establish a Leibniz formula for higher order Dunkl derivatives which plays a crucial role in proving one of our main results. For two reasonable functions $f$ and $g$, if  one of them is invariant under the action of $G$, then we have the following Leibniz-type rule is known: 
	\begin{equation}\label{leib-1} \mathcal{D}_j(fg)(x)= \mathcal{D}_jf(x)\, g(x)+f(x) \mathcal{D}_j g(x).
	\end{equation}
	Even if $ f $ is $G$-invariant, its Dunkl derivative $ \mathcal{D}_j f $ need not be $G$-invariant for which it is difficult to come up with a Leibniz rule for higher order derivatives. In order to prove a classical version of multiplier theorem for the Dunkl transform which involves conditions on the derivatives of the multiplier $ m $ we need a usable formula for higher order derivatives of $ m(\xi) |\xi|^{2k} e^{-t|\xi|^2}.$ More precisely we would like to know if the estimates $ |\mathcal{D}^\alpha m(\xi)| \leq C\, |\xi|^{-|\alpha|} $ lead to same estimates  for derivatives of the function 
	$m(\xi) |\xi|^{2k} e^{-t|\xi|^2}.$ To this end we obtain a Leibniz formula for functions of the form $ mg $ where $ g $ is radial.\\
	
	Leibniz formula for $ m g $ where $ g $ is radial involves repeated application of certain averaging operators $ A_\nu $ which we define now.  For any function $\nu: R\,\cup\, \{0\}\to \mathbb{C}$, the operator $A_\nu$ is  given by 
	$$ A_\nu\, m(x)=\sum\limits_{\lambda \in R\,\cup\, \{0\}}\nu(\lambda)\, m(\sigma_\lambda x)$$
	where $\sigma_0=I$. In other words
	$ A_\nu\, m=\sum\limits_{\lambda \in R\,\cup\, \{0\}}\nu(\lambda)\, \sigma_\lambda m.$ When the function $\nu$ depends on a multi-index $\alpha$, we will sometimes denote the corresponding operator by $A_{\nu_{\alpha}}$. For example, when $\nu_\alpha(\lambda) =  \nu(\lambda)\,\prod_{j=1}^d(c_j\lambda_j)^{\alpha_j}$, where $\nu: R\,\cup\, \{0\}\to \mathbb{C}$ is an arbitrary function, we write $A_{\nu_\alpha}$ to denote the operator given by $A_{\nu_\alpha}\, m(x)=\sum_{\lambda \in R\,\cup\, \{0\}}\nu(\lambda)\, \prod_{j=1}^d(c_j\lambda_j)^{\alpha_j} m(\sigma_\lambda x)$. Below we record  some easy to prove properties of $A_\nu$ that will be utilized later.\\
	
	\begin{lem}\label{lem-1} (i) If $m \in C^N(\mathbb{R}^d \setminus \{0\})$ satisfies $|\mathcal{D}^\alpha m(x)| \lesssim |x|^{-|\alpha|}$ for all multi-indices $|\alpha| \leq N$, then $A_\nu\, m$ also satisfies the same condition. 
		
		\noindent
		(ii) For any $ \alpha \in \mathbb N^d,$  there are functions $\nu_\beta $ such that
		$$A_\nu\, (x^\alpha\, m)=\sum\limits_{|\beta|=|\alpha|} x^\beta A_{\nu_{\beta}} m.$$
	\end{lem}
	\begin{proof} We make use of the property  $ \sigma \circ \mathcal{D}_\xi \circ \sigma^{-1} = \mathcal{D}_{\sigma(\xi)} $  valid for any $ \sigma \in G.$  In particular, for any $ \sigma_\lambda $ we have (as $\sigma_\lambda =  \sigma_\lambda^{-1}$)
		\begin{equation}\label{eq-1}(\mathcal{D}_j\circ \sigma_\lambda)m = (\sigma_\lambda \circ \mathcal{D}_{\sigma_\lambda(e_j)})m = \sum_{\ell=1}^d e_\ell \cdot \sigma_\lambda(e_j)\, (\sigma_\lambda \circ \mathcal{D}_\ell )m.
		\end{equation}
		From this it is clear that if $ \mathcal{D}_jm $ has certain decay, the same thing is true for $ (\mathcal{D}_j\circ \sigma_\lambda)m.$
		As  the above is true for any $ 1 \leq j \leq d, $  we have the formula
		$$ (\mathcal{D}_j  \circ A_\nu)m =  \sum\limits_{\lambda\in R\,\cup\, \{0\}}\nu(\lambda) \, ( \mathcal{D}_j \circ \sigma_\lambda)m = \sum_{\ell =1}^d  \sum\limits_{\lambda\in R\,\cup\, \{0\}} \nu (\lambda)\,( e_\ell \cdot \sigma_\lambda(e_j))\, (\sigma_\lambda \circ \mathcal{D}_\ell )m .$$
		In other words, $  \mathcal{D}_j \circ A_\nu = \sum_{\ell=1}^d A_{\nu_{j,\ell}} \circ \mathcal{D}_\ell $ where $ \nu_{j,\ell}(\lambda) = \nu(\lambda) \, (e_\ell \cdot \sigma_\lambda(e_j)) $ and
		hence  our claim is true for $ |\alpha| =1.$ For higher order derivatives  we just need to use straightforward induction.
		The second property stated in the lemma is easy to verify. From the definition
		$$ A_\nu\, (x^\alpha\, m)(x) = \sum\limits_{\lambda \in R\,\cup\, \{0\}}\nu(\lambda)\,\prod\limits_{j=1}^d \big(x_j-\langle x, \lambda \rangle\lambda_j\big)^{\alpha_j} m(\sigma_\lambda x) $$
		which after simplification reduces to $ \sum\limits_{|\beta|=|\alpha|} x^\beta A_{\nu_{\beta}} \,m (x).$
	\end{proof}
	
	The averaging operators $ A_\nu $ occur when we try to compute the commutator of $ \mathcal{D}_j $ with the operator of multiplication by the monomial $ x^\alpha.$  Consider first
	$$ \mathcal{D}_j\,(x_{\ell}\,m)(x) = \delta_{j, \,\ell}\,m +x_{\ell}\,\partial_jm + x_{\ell}\sum\limits_{\lambda \in R}\frac{\K(\lambda)}{2}\lambda_j \frac{m(x)-m(\sigma_\lambda x)}{\langle \lambda , x \rangle} +\sum\limits_{\lambda \in R}\frac{\K(\lambda)}{2}\lambda_j\, \lambda_{\ell}\, m(\sigma_\lambda x) .$$
	Thus we see that 
	\begin{equation}\label{commut} \mathcal{D}_j (x_{\ell}\, m) - x_{\ell}\, \mathcal{D}_j m = \delta_{j, \,{\ell}}\,m +\sum\limits_{\lambda \in R}\frac{\K(\lambda)}{2}\lambda_j\, \lambda_{\ell}\, m(\sigma_\lambda x) = A_{\nu}\, m,
	\end{equation}
	where $\nu: R\,\cup\, \{0\}\to \mathbb{C}$ is given by $\nu (\lambda) =
		\left\{
		\begin{array}{ll}
			\delta_{j, \,{\ell}}  & \mbox{if } \lambda = 0, \\
			\frac{\K(\lambda)\lambda_j\, \lambda_{\ell}}{2} & \mbox{if } \lambda \in R.
		\end{array}
		\right.$\\
	\\
	
	Using induction we can obtain a formula for the commutator of $ \mathcal{D}_j $ with the operator of multiplication by the monomial $ x^\alpha.$
	
	\begin{lem}\label{lem-2} For any $ C^1 $ function $m$ and  multi-index $\alpha$  we have the commutator relation
		$$			\mathcal{D}_j(x^\alpha \,m) - x^\alpha \,\mathcal{D}_jm = \sum\limits_{|\beta|=|\alpha|-1}  x^\beta \,A_{\nu_j(\beta)}\,m .$$
		for some functions $\nu_j(\beta).$
	\end{lem}
	\begin{proof} We have already seen that the result is true when $ |\alpha|=1.$ Assuming the result for some $ \alpha $ let us consider
		$$ \mathcal{D}_j(x^{\alpha +e_k}\,m)= \mathcal{D}_j(x_kx^\alpha \,m)= x_k\, \mathcal{D}_j(x^\alpha \,m) + A_\nu(x^\alpha \,m)$$
		where in the second equality we have used \eqref{commut}. By the induction hypothesis and part (ii) of the previous lemma we have
		$$ \mathcal{D}_j(x^{\alpha +e_k}\,m) = x_k \, \Big(x^\alpha \,\mathcal{D}_jm +\sum\limits_{|\beta|=|\alpha|-1}  x^\beta \,A_{\nu_j(\beta)}\,m \Big)+ \Big(\sum\limits_{|\beta|=|\alpha|} x^\beta A_{\nu_{\beta}} \,m \Big).$$
		Renaming $ \beta+e_k $ in the first summand as $ \beta,$  we complete the induction. The same computation works for any $ k.$
	\end{proof}
	
	As we are interested a Leibniz formula for $ \mathcal{D}^\alpha(m\, g) $ where $ g $ is radial, we start with the simplest case. By Leibniz formula for $ \mathcal{D}_j $ we have
	$$ \mathcal{D}_j(m\, g) = g \, \mathcal{D}_j m\,  + m\, \partial_j g = g \, \mathcal{D}_j m\,  + x_j\, m\,  D_r h$$
	where $ g(x) = h(|x|) $ and $ D_r = r^{-1}  \frac{d}{dr}.$ When we compute the second order derivatives $ \mathcal{D}_k \mathcal{D}_j (m\,g) $ we encounter terms of the form 
	$\mathcal{D}_k(x_j\, m) $ which can be handled using the first order commutation relations proved in \eqref{commut}. For derivatives of higher order more complicated terms occur. For example, for $ \mathcal{D}_{\ell} \mathcal{D}_k \mathcal{D}_j (m\,g) $ we need to handle terms of the form  $\mathcal{D}_\ell \, ( x_k\, \mathcal{D}_j m).$ When we use commutation relations to move the factor $ x_k $ to the left of $ \mathcal{D}_\ell$ terms of the form $ A_\nu( \mathcal{D}_j m) $ occur. To handle these kind of terms we use the following result.\\
	
	\begin{lem}\label{lem-3} For any $ \alpha \in \mathbb N^d $ there are constants $ c_{\alpha,\beta}$ and functions $\nu_{\alpha,\beta} $ so that we have 
		$$ A_\nu \circ \mathcal{D}^\alpha  = \sum_{|\beta|=|\alpha|} \, c_{\alpha,\beta} \, \mathcal{D}^\beta \circ A_{\nu_{\alpha,\beta}}.$$
	\end{lem}	
	\begin{proof} We start with the simplest case  $ \alpha = e_j.$  Consider 
		$$ (A_\nu \circ \mathcal{D}_j) m = \sum\limits_{\lambda\in R\,\cup\, \{0\}}\nu(\lambda) \, (\sigma_\lambda \circ \mathcal{D}_j)m  .$$
		Since $ \sigma_\lambda \circ \mathcal{D}_j = \mathcal{D}_{\sigma_\lambda(e_j)}\circ \sigma_\lambda, $ recalling that  $\mathcal{D}_{\sigma_\lambda(e_j)}=\sum_{\ell=1}^d (e_\ell \cdot \sigma_\lambda(e_j))\mathcal{D}_\ell$, we can rewrite the above as 
		$$(A_\nu \circ \mathcal{D}_j) m =  \sum_{\ell =1}^d \sum\limits_{\lambda\in R\,\cup\, \{0\}}\nu(\lambda) \, (e_\ell \cdot \sigma_\lambda(e_j)) ({\mathcal{D}_\ell} \circ \sigma_\lambda)m .$$
		By defining $ \nu_{j,\ell}(\lambda) = \nu(\lambda)\, (e_\ell \cdot \sigma_\lambda(e_j)) $ we get the following formula:
		\begin{equation}\label{eq-2} (A_\nu \circ \mathcal{D}_j) m =  \sum_{\ell =1}^d  \sum\limits_{\lambda\in R\,\cup\, \{0\}}\,  \,\nu_{j,\ell}(\lambda) ({\mathcal{D}_\ell} \circ \sigma_\lambda)m =  \sum_{\ell =1}^d  
		( \mathcal{D}_\ell \circ A_{\nu_{j,\ell}})m.
		\end{equation}
		This proves the lemma when $ \alpha = e_j.$ By replacing $ m $ by $ \mathcal{D}_k m$  we see that
		$$ (A_\nu \circ \mathcal{D}_j \circ \mathcal{D}_k)m = \sum_{\ell =1}^d  ( \mathcal{D}_\ell \circ A_{\nu_{j,\ell}} \circ \mathcal{D}_k)m.$$ 
		Using \eqref{eq-2}  for the operator $ A_{\nu_{j,\ell}} \circ \mathcal{D}_k$ we obtain
		$$  ( \mathcal{D}_\ell \circ A_{\nu_{j,\ell}} \circ \mathcal{D}_k)m = \sum_{r =1}^d ( \mathcal{D}_\ell \circ  \mathcal{D}_r \circ A_{\nu_{j,\ell,k,r}} )m.$$
		It is now clear that lemma follows by induction on $ |\alpha|.$
		
	\end{proof}

	Finally, we are in a position to state and prove a Leibniz formula for $ \mathcal{D}^\alpha (m\,g)$ were $ g(x) = h(|x|) $ is radial. Recall that $ D_r = r^{-1} \frac{d}{dr}.$

	\begin{prop}\label{Leibniz}
		Let $m$ and $g$ be two reasonable functions with $g$ being radial. Then for any multi-index $\alpha$, we have the following Leibniz type rule
		\begin{equation}\label{leib}
			\mathcal{D}^{\alpha}(m\,g)=\sum\limits_{\ell =1}^N \sum\limits_{2k+|\gamma|-|\beta|=|\alpha|} c_{\alpha, \beta,\, \gamma,\, k}\,x^{\beta}\, (\mathcal{D}^{\gamma} \,m_\ell)\, (D_r^k h),
		\end{equation}
		for some  $ N \in \mathbb N $ where $ m_\ell = (A_{\nu_1} \circ A_{\nu_2}\circ...\circ A_{\nu_{n_\ell}}) m$ for some functions $ \nu_j.$ Here $ N $ depends on $ \alpha$ and $ n_\ell $ depends on the triple $ (\beta, \gamma, k).$
	\end{prop}
	
	\begin{proof} As usual we prove this proposition by induction. We have already checked that
		$$ \mathcal{D}_j(m\, g) = g \, \mathcal{D}_j m\,  + m\, \partial_j g = g \, \mathcal{D}_j m\,  + x_j\, m\,  D_r h$$
		which proves the result when $ |\alpha| =1.$  Assume \eqref{leib} for certain $ \alpha $ and consider
		$$\mathcal{D}^{\alpha+e_j}(m\,g)=\sum\limits_{\ell =1}^N \sum\limits_{2k+|\gamma|-|\beta|=|\alpha|} c_{\alpha, \beta,\, \gamma,\, k} \,\, \mathcal{D}_j \big(x^{\beta}\, \mathcal{D}^{\gamma} \,m_\ell\, D_r^k h \big).$$
		As $D_r^k h $ is radial, Leibniz formula for $ \mathcal{D}_j $ gives 
		$$\mathcal{D}_j \big(x^{\beta}\, \mathcal{D}^{\gamma} \,m_\ell\, D_r^k h \big)=  \mathcal{D}_j  \big(x^{\beta}\, \mathcal{D}^{\gamma} \,m_\ell\, \big) D_r^kh + x_j\,x^{\beta}\, \mathcal{D}^{\gamma} \,m_\ell\,
		D_r^{k+1}h.$$
		The contribution of the second term to $\mathcal{D}^{\alpha+e_j}(m\,g) $ is therefore a finite sum of terms of the form
		$$\sum\limits_{2k+|\gamma|-|\beta|=|\alpha|} c_{\alpha, \beta,\, \gamma,\, k} \,\,x^{\beta+e_j}\, \mathcal{D}^{\gamma} \,m_\ell\,
		D_r^{k+1}h = \sum\limits_{2k+|\gamma|-|\beta|=|\alpha+e_j|} c_{\alpha, \beta -e_j,\, \gamma,\, k-1} \,\,x^{\beta}\, \mathcal{D}^{\gamma} \,m_\ell\,
		D_r^{k}h .$$
		To handle the contribution of the first term we make use of the results of the lemmas proved above. In view of Lemma \ref{lem-2} 
		$$\mathcal{D}_j  \big(x^{\beta}\, \mathcal{D}^{\gamma} \,m_\ell\, \big) D_r^kh  =  \Big( x^\beta \, \mathcal{D}^{\gamma+e_j} \,m_\ell\, + \sum\limits_{|\beta^\prime|=|\beta|-1}  x^{\beta^\prime} \,(A_{\nu_j(\beta^\prime)}\circ \mathcal{D}^{\gamma}) \,m_\ell\, \Big) \,D_r^kh .$$
		Once again the first term is of the form:
		$$\sum\limits_{2k+|\gamma|-|\beta|=|\alpha|} c_{\alpha, \beta,\, \gamma,\, k} \,\,x^{\beta}\, \mathcal{D}^{\gamma+e_j} \,m_\ell\,
		D_r^{k}h = \sum\limits_{2k+|\gamma|-|\beta|=|\alpha+e_j|} c_{\alpha, \beta,\, \gamma -e_j,\, k} \,\,x^{\beta}\, \mathcal{D}^{\gamma} m_\ell\,
		D_r^{k}h .$$
		We are now left with the following sum:
		$$\sum\limits_{2k+|\gamma|-|\beta|=|\alpha|}  \sum\limits_{|\beta^\prime|=|\beta|-1} c_{\alpha, \beta,\, \gamma,\, k}\, \,x^{\beta^\prime} \big(A_{\nu_j(\beta^\prime)}\circ \mathcal{D}^{\gamma}\big) m_\ell\,  \,D_r^kh .$$
		We now use the result of Lemma \ref{lem-3} to rewrite $A_{\nu_j(\beta^\prime)}\,\mathcal{D}^{\gamma} m_\ell $  as a finite sum of terms of the form
		$$  \big( \mathcal{D}^{\gamma^\prime} \circ A_{\nu_j(\beta^\prime,\gamma^\prime)} \big)m_\ell ,\text{ where } |\gamma^\prime| = |\gamma|.$$
		Recalling that $ m_\ell = (A_{\nu_1} \circ A_{\nu_2}\circ...\circ A_{\nu_{n_\ell}}) m$ we end up with sums of the form
		$$\sum\limits_{2k+|\gamma^\prime|-|\beta|=|\alpha|}  \sum\limits_{|\beta^\prime|=|\beta|-1} c_{\alpha, \beta,\, \gamma^\prime,\, k}\, \,x^{\beta^\prime} \,\mathcal{D}^{\gamma^\prime} m_{\ell^\prime} \,  \,D_r^kh .$$
		Clearly each of these are in the required form completing our induction hypothesis.
	\end{proof}

	When both $ m $ and $ g $ are radial the formula in the above proposition takes a simple form. First of all observe that when $ m $ is radial $ A_\nu m = c_\nu m $ where $ c_\nu $ is a constant. Thus from from Proposition \ref{Leibniz} we see that
	$ \mathcal{D}^\alpha (m g) $ is a linear combination of terms of the form $ x^\beta\, \mathcal{D}^{\gamma} m\, D_r^k h.$ We also note that when $ m(x) = m_0(|x|), \mathcal{D}_j m(x) = x_j\, D_rm_0(|x|).$ This leads to the following simplified
	Leibniz rule.
	
	\begin{cor}\label{rad-Leibnitz} Let $ m(x) = m_0(|x|),\, g(x) =  h(|x|) $  be two reasonable functions.  Then for any multi-index $\alpha$, we have the following Leibniz type rule
		\begin{equation}\label{rad-leib}
			\mathcal{D}^{\alpha}(m\,g)(x)= \sum\limits_{k+\ell \leq |\alpha|, 2(k+\ell)-j=|\alpha|} P_j(x) \,D_r^k \,m_0(|x|)\, D_r^\ell h(|x|)
		\end{equation}
		where for each $ j, P_j(x) $ is a homogeneous polynomial of degree $j.$
	\end{cor}
	\begin{proof} From the Leibniz formula 
		$$ \mathcal{D}_j(m\, g) = g \, \partial_j m\,  + m\, \partial_j g = x_j( g \, D_r m_0\,  +  m_0\,  D_r h) $$
		which shows that the corollary is true for $ |\alpha|=1.$ Assuming the formula for a certain $ \alpha $ we consider 
		\begin{equation}\label{rad-leib1}
			\mathcal{D}^{\alpha+e_1}(m\,g)= \sum\limits_{k+\ell \leq |\alpha|, 2(k+\ell)-j=|\alpha|} \mathcal{D}_1 \Big( P_j \,D_r^k \,m_0\, D_r^\ell h\Big).
		\end{equation}
		As $ P_j $ is homogeneous of degree $ j $ we have		
		$$ \mathcal{D}_1 \Big( P_j \,D_r^k \,m_0\, D_r^\ell h\Big) =\sum_{|\beta|=j} c_\beta  \mathcal{D}_1 \Big( x^\beta \,D_r^k \,m_0\, D_r^\ell h\Big).$$
		Using the result of Lemma \ref{lem-2} and noting that $ D_r^k \,m_0\, D_r^\ell h $ is radial we have
		$$ \mathcal{D}_1 \Big(x^\beta\, D_r^k \,m_0\, D_r^\ell h\Big) = x^{\beta+e_1} \, D_r\Big( D_r^k \,m_0\, D_r^\ell h\Big) + Q_{\beta,j}(x) \Big(D_r^k \,m_0\, D_r^\ell h\Big)$$
		where $ Q_{\beta, j} $ is a homogeneous polynomial of degree $(j-1).$  Noting that  
		$$ D_r \Big( D_r^k \,m_0\, D_r^\ell h\Big) = \Big( D_r^{k+1} \,m_0\, D_r^\ell h + D_r^k m_0\, D_r^{\ell+1}h \Big),$$
		and putting all the terms together  we can check that the $ \mathcal{D}^{\alpha+e_1}(mg) $ is in the required form. This completes the induction.
	\end{proof}

	\section{Revisiting Fourier multipliers}\label{sec-revisit Fourier multipliers}
	An efficient and elegant approach to establishing the Fourier multiplier theorem is through the use of Littlewood-Paley square functions, as thoroughly presented in Stein's book \cite{Stein}. In this section we briefly recall the theory of $g$-functions and describe how it can be used to prove a version of H\"ormander-Mihlin multiplier theorem for the Fourier transform.
	
	\subsection{Littlewood-Paley square functions} Given a measure space $ (X, \mathcal{B},\mu) $ let $ L^p(X,d\mu) $ be the associated $ L^p $ spaces. Recall that   $ \{T_t \}_{t\geq 0}$ is a symmetric diffusion semigroup acting on $ L^p(X,d\mu), 1 \leq p \leq \infty  $ if they are contractions on $ L^p(x, d\mu),$ self-adjoint  on $ L^2(X, d\mu) $ such that $ T_tf \geq 0$ for $ f \geq 0 $ and $ T_t1 =1.$ Given such a semigroup, we define the $ g$-functions by
	$$ g_k(f,x) = \left( \int_0^\infty |\partial_t^k T_tf(x)|^2\, t^{2k-1}\, dt\right)^{1/2}.$$
	Using theory of martingales, Stein \cite{Stein} has proved that $ \| g_k(f)\| \leq C\|f\|_p, 1 < p < \infty $ and under the extra assumption that $ \| g_1(f)\|_2 = c \|f\|_2$ we have the equivalence of norms:
	$$  A\, \|f \|_p  \leq \|g_k(f)\|_p \leq B \, \|f\|_p.$$
	Let us specialise this to the case when $ X =\R^n $ and $ \{T_t \}_{t\geq 0} $ is the heat semigroup defined by 
	$$ T_tf(x) = f \ast p_t(x) = (2\pi)^{-n/2} \int_{\R^n} e^{i x \cdot \xi}\, \widehat{f}(\xi)\, e^{-t|\xi|^2}\, d\xi $$
	where $ p_t(x) $ is the heat kernel associated to the Laplacian $ \Delta$ on $ \R^n.$ 
	
	In this case the $ g_k $-functions can be defined even for non-integer values of $ k.$ For any $ s > 0 $ let us define
	$$ g_s(f,x) = \left( \int_0^\infty |G_sf(x,t)|^2 t^{2s-1}\, dt \right)^{1/2}$$
	where we have set 
	$$ G_sf(x,t) = (2\pi)^{-n/2} \int_{\R^n} e^{i x\cdot \xi}\, \widehat{f}(\xi)\,  |\xi|^{2s} e^{-t|\xi|^2}\, d\xi .$$
	Observe that  when $ s = k $ is an integer, $ G_sf(x,t) = \partial_t^kT_tf(x) $ so that we get back the $g_k$ function defined earlier. For these square functions we can prove the following result. Since the general theory does not apply ( when $ s $ is not an integer) we have a different proof.
	
	\begin{thm}\label{thm 1st} Let $ s > 1/2 $ and  $ 1 < p < \infty.$ Then for all $ f \in L^p(\R^n) $
		$$  A\, \|f \|_p  \leq \|g_s(f)\|_p \leq B \, \|f\|_p.$$
	\end{thm}
	\begin{proof} By Plancherel, it immediately follows that $ \|g_s(f)\|_2 = 2^{-s} \sqrt{\Gamma(2s)} \| f \|_2.$ Therefore, it is enough to show that $ f \rightarrow g_s(f) $ is bounded on $ L^p(\R^n).$ To prove this, we note that $ G_s: f \rightarrow G_sf(x,t) $ is a convolution operator, viz. $ G_sf(x,t) = f \ast K_s(x,t) $ where
		$$ K_s(x,t) = \int_{\R^n} e^{i x\cdot \xi}\,  |\xi|^{2s} e^{-t|\xi|^2}\, d\xi .$$
		If we can show that $ K_s(x,t) $ is a Calderon-Zygmund kernel taking values in the Hilbert space $ \mathcal{H}_s = L^2(\R^+,t^{2s-1}dt), $ then we can appeal to the theory of vector-valued singular integral operator to prove the boundedness of the $g_s$-function.
		Now the kernel $ K_s(x,t) $ is clearly radial and satisfies $ K_s(x,t) = t^{-n/2-s} K_s(x/\sqrt{t},1) = t^{-n/2-s} k_s(|x|/\sqrt{t})$, where $k_s(|x|):=K_s(x, 1)$. Therefore, we calculate that
		$$ \int_0^\infty |K_s(x,t)|^2 t^{2s-1} dt = \int_0^\infty  t^{-n-1} |k_s(|x|/\sqrt{t})|^2 dt = C_s\, |x|^{-2n} $$ 
		where the constant $ C_s $ is given by the integral
		$$  2 \int_0^\infty  t^{2n-1} |k_s(t)|^2 dt  = c_n \int_{\R^n} |x|^n \, |K_s(x,1)|^2 dx. $$
		It is easy to see that this integral is finite. Indeed, integrating by parts in the integral defining $ K_s(x,1)$ we have
		$$ x^\alpha K_s(x,1) = i^{|\alpha|} \int_{\R^n} e^{i x\cdot \xi} \, \partial^\alpha( |\xi|^{2s} e^{-|\xi|^2}) \, d\xi.$$
		As we are assuming $ s > 1/2$ the above gives the estimate $ |K_s(x,1)| \leq C_\alpha |x|^{-|\alpha|} $ as long as $ |\alpha | \leq n+1.$
		Hence $ C_s < \infty $ and we get the estimate 
		$ \| K_s(x,\cdot) \|_{\mathcal{H}_s}  \leq C |x|^{-n}.$ In a similar  way we can show that $ \| \nabla K_s(x,\cdot)\|_{\mathcal{H}_s} \leq C_s |x|^{-n-1}.$
	\end{proof}
	
	Along with $g$-functions we also need $g^\ast$-functions which are defined as follows.
	$$ g_s^\ast(f,x)^2 = \int_0^\infty \left( \int_{\R^n} |\partial_t T_tf(y)|^2 (1+t^{-1}|x-y|^2)^{-s} dy\right) t^{-n/2+1} dt.$$

	\begin{thm}\label{thm 2nd} Let $ s > n/2 .$ Then for any $ 2 < p < \infty $ we have $ \| g_s^\ast(f) \|_p \leq C \|f\|_p.$
	\end{thm}
	
	This result is proved by using the properties of the $g_1 $ function and  the boundedness of the Hardy-Littlewood maximal function. 
	
	\subsection{A version of H\"ormander-Mihlin multiplier theorem} The boundedness of the Fourier multiplier $ T_m $ on $ L^p(\R^n) $ is proved under the following condition on the multiplier. For any $ \psi \in C_0^\infty(\R^n) $ supported in the complement of $ 0$ define 
	$$  \widehat{m}_\psi(x,t) = \int_{\R^n} e^{i x \cdot \xi} \psi(\xi) m( \xi/\sqrt{t}) d\xi.$$ 
	We say that $ m $ satisfies  H\"ormander's condition of order $ s $ if  for any such $ \psi$
	$$ \int_{\R^n} (1+|x|^2)^s | \widehat{m}_\psi(x,t)|^2 dx \leq C_\psi $$ 
	for all $ t> 0.$  By replacing $ \psi $ by $ |\cdot|^{2s} \psi $ we can rewrite the above condition as follows. Let
	$$ \widetilde{m}_{\psi,s}(x ,t) = \int_{\R^n} e^{i x \cdot \xi} |\xi|^{2s} \psi(\sqrt{t} \xi) m( \xi) d\xi.$$ 
	Then the above condition on $ m $ is equivalent to the estimate 
	$$\int_{\R^n} (1+t^{-1}|x|^2)^s |\widetilde{m}_{\psi,s}(x,t)|^2\, dx \leq C_{\psi,s}\,  t^{-2s-n/2}.$$
	Replacing the function $ \psi $ by the Gaussian $ e^{-|\xi|^2} $ in the above  let us define
	$$ \widehat{m}_s(x,t)  =\int_{\R^n} e^{i x \cdot \xi} |\xi|^{2s}e^{-|\xi|^2} m( \xi/\sqrt{t}) d\xi.$$  
	We say that $ m $ satisfies  the modified H\"ormander's condition of order $ s $ if  
	\begin{equation}\label{modH}  \sup_{t >0} \int_{\R^n} (1+|x|^2)^s | \widehat{m}_s(x,t)|^2 dx \leq C_s.
	\end{equation}
	Clearly, this condition is slightly stronger than H\"ormander's condition. To see this, we note that in the H\"ormander's condition we can replace the function $\psi$ by $\psi_1$, where $\psi_1(\xi)=\psi (\xi)\, |\xi|^{2s}$. Then we can write $\widehat{m}_{\psi_1}(x,t)=\psi_2 \ast \widehat{m}_s(\cdot,t)(x),$ where $\widehat{\psi}_2(\xi)=\psi(\xi)\,e^{|\xi|^2}$. Now a straightforward calculation using Minkowski's integral inequality shows that for any $t>0$, 
		$$\int_{\R^n} (1+|x|^2)^s | \widehat{m}_{\psi_1}(x,t)|^2 dx \leq C_{\psi_2,\, s} \int_{\R^n} (1+|x|^2)^s | \widehat{m}_s(x,t)|^2 dx.$$
	As before by defining
	$$ \widetilde{m}_s(x,t) = \int_{\R^n} e^{i x\cdot \xi} \, |\xi|^{2s} m(\xi) e^{-t|\xi|^2} \, d\xi $$ 
	we can show that the modified H\"ormander's condition is equivalent to
	\begin{equation}\label{modified HM}
		\int_{\R^n} (1+t^{-1}|x|^2)^s |\widetilde{m}_s(x,t)|^2\, dx \leq C\, t^{-2s-n/2}.
	\end{equation}
	Using the properties of the $g$ and $ g^\ast$ functions we can easily prove a version of multiplier theorem under a modified H\"ormander condition.
	
	\begin{thm} Let $ m $ be a bounded function which satisfies the condition \eqref{modified HM} for some $ s> n/2.$ Then the Fourier multiplier $ T_m $ is bounded on $ L^p(\R^n) $ for any $ 1 < p < \infty.$
	\end{thm}
	\begin{proof} In view of Theorem \ref{thm 1st} we have $ \|T_mf\|_p \leq C \| g_{s+1}(T_mf)\|_p.$ If we can show that  $ g_{s+1}(T_mf,x) \leq C g_s^\ast(f,x) $ pointwise, then by Theorem \ref{thm 2nd} we get $ \|T_mf\|_p \leq C \|f\|_p $ for any $ p >2.$ For $ 1< p < 2$  we can use duality. Recalling the definition $ G_{s+1}f(x,t) $ and $ \widetilde{m}_s(x,t) $ we can write
		$$ G_{s+1}T_mf(x,2t) = \int_{\R^n} \partial_t T_tf(x-y)\, (1+t^{-1}|y|^2)^{-s/2}  (1+t^{-1}|y|^2)^{s/2} \widetilde{m}_s(y,t)\, dy.$$
		Under the assumption that $ m $ satisfies \eqref{modified HM} we get
		$$ |G_{s+1}T_mf(x,2t)|^2 \leq C t^{-2s-n/2}  \int_{\R^n} |\partial_t T_tf(x-y)|^2\, (1+t^{-1}|y|^2)^{-s} \, dy.$$
		Integrating the above against $ t^{2s+1} dt $ we obtain  $ g_{s+1}(T_mf,x) \leq C g_s^\ast(f,x) $.
	\end{proof}
	
	\begin{rem} Consider the following variant of H\"ormander condition 
		$$ \left( \int_{R \leq |\xi| \leq 2R } | \partial^\alpha m(\xi)|^2 \, d\xi \right)^{1/2} \leq  C R^{-|\alpha|+\frac{1}{2}n}.$$
		We claim that this is stronger than the condition \eqref{modH} when $ s = k $ is an integer. To see that we need to estimate the integral
		$$ \int_{\R^n}  |(1-\Delta)^{k/2}\left(|\xi|^{2k}e^{-|\xi|^2} m( \xi/\sqrt{t})\right)|^2 d\xi.$$  
		Since Sobolev spaces of integer order can be described in terms of derivatives, it is enough to estimate
		$$ \int_{\R^n}  |\partial^\alpha\left(|\xi|^{2k}e^{-|\xi|^2} m( \xi/\sqrt{t})\right)|^2 d\xi$$ 
		for all $ |\alpha|\leq k.$ Consider the term when all the derivatives fall on $ m$ which  is given by
		$$ t^{-|\alpha|+2k+n/2} \int_{\R^n}  |\xi|^{4k}e^{-2t|\xi|^2}  |\partial^\alpha m(\xi)|^2 d\xi. $$ 
		To estimate the above integral we note that
		$$ \int_{\R^n}  |\xi|^{4k}e^{-2t|\xi|^2}  |\partial^\alpha m(\xi)|^2 d\xi \leq  C \sum_{j=-\infty}^\infty 2^{2kj} e^{-t 2^{j+1}}  \int_{ 2^{j/2} \leq |\xi| \leq 2^{(j+1)/2}}  |\partial^\alpha m(\xi)|^2 d\xi.$$ 
		As $ m $ is assumed to satisfy the H\"ormander condition, we get the estimate
		$$ t^{-|\alpha|+2k+n/2} \int_{\R^n}  |\xi|^{4k}e^{-2t|\xi|^2}  |\partial^\alpha m(\xi)|^2 d\xi \leq C\sum_{j=-\infty}^\infty \left(t\,2^j\right)^{(-|\alpha|+n/2+2k)} \, e^{-t \,2^{j+1}} .$$
		The above series converges for all $ t > 0 $ and defines a locally bounded function $ \psi(t) $ which satisfies $ \psi(t) = \psi(2t).$ Hence $ \psi(t) $ is a bounded function. In a similar way we can estimate the other terms as well.
	\end{rem}

	\section{Multiplier theorems for the Dunkl transform}\label{sec-multiplier thm for Dunkl trans}
	In this section we use Littlewood-Paley-Stein theory of $g$-functions defined by the semigroup $\{T_t^\delta\}_{t\geq 0}= \{e^{-t(-\Dk)^\delta}\}_{t\geq 0} $ and prove multiplier theorems. When $ m $ is radial we recover the known result due to Dai and Wang \cite{WangATT}. For the general case we prove the multiplier theorem for radial functions in $ L^p(\R^d,h_\K^2)$ but only for $ p \geq 2.$
	
	\subsection{Littlewood-Paley square functions in the Dunkl setting}  If we let $ \{T_t \}_{t\geq 0} $ stand for the heat semigroup generated by the Dunkl Laplacian $ \Delta_{\kappa} $ then general theory applies and we get the boundedness of $ g_k , k \in \mathbb N$ and $ g_s^\ast , s >d_\K/2 $  as in the case of Fourier transform. However, if we want to prove the boundedness of $ g_s$-functions for non integral values of $ s $, then there is no readily available singular integral theory for convolution type operators with optimal conditions. The main obstacle is the lack of knowledge about the $ L^p$ boundedness of the Dunkl translation operator $ \tau_x.$ There are convolution type singular integral theories developed in the Dunkl setting (see, for example, \cite{HejnaSIIRDS, HejnaUALBFLPSF, LiDFSFEFDO, S, S1}) but it is difficult verify the required strong assumptions on the kernel. For example we need pointwise estimates for the translation of the kernel
	$$ K_s(x,t) = \int_{\R^d} E_\K (ix,\xi) |\xi|^{2s} e^{-t|\xi|^2}\, h^2_\kappa(\xi) \, d\xi.$$
	When $ s=k $ is a positive integer, $ \tau_\K(x)K_s(y,t) = \partial_t^k h_t(x,y)$ where $ h_t(x,y) $ is the Dunkl heat kernel. Hence in this case we can use known estimates of $ h_t(x,y) $ and its derivatives $\partial_{x}^{\alpha}h_t(x,y)$ and $\partial_{y}^{\alpha}h_t(x,y)$ (see \cite[Theorem 4.1]{HejnaHFCH}). However, when $ s $ is not an integer we do not know how to estimate $ \tau_\K(x)K_s(y,t).$
	
	When $ s > 0 $ is not an integer, we choose $ k $ as the least integer  such that $ s <k $ and let $ \delta = s/k <1.$ Instead of the semigroup $ \{T_t\}_{t\geq 0} =\{ e^{-t(-\Delta_\kappa)}\}_{t\geq 0} $ we consider the semigroup $ \{T_{t,\delta}\}_{t\geq 0} =\{e^{-t(-\Delta_\kappa)^\delta}\}_{t\geq 0} $ generated by the fractional power $ (-\Delta_\kappa)^\delta.$
	From the above representation \eqref{boch-sub} it is clear that $\{T_{t,\delta}\}_{t\geq 0} $ is a symmetric diffusion semigroup in the sense of Stein \cite{Stein} and by general theory we have the following result.  Let $ g_{j,\delta}$ be the $ g $ function associated to $ T_{t,\delta} $ and for $ s> 0$ define
	$$ g_{s,\delta}^\ast(f,x)^2 = \int_0^\infty \left( \int_{\R^d} |\partial_t T_{t,\delta}f(y)|^2 \tau_\K(x)(1+t^{-\frac{1}{\delta}}|-y|^2)^{-s} h^2_\kappa(y) dy\right) t^{-\frac{d_\K}{2\delta}+1} dt.$$
	
	\begin{thm} For  any  $ 1 < p < \infty,$ we have $ A\, \|f\|_p \leq \|g_{j,\delta}(f)\|_p \leq B\, \|f\|_p$  for all $ f \in L^p(\R^d, h_\kappa^2 ).$ If $ s >d_\K/2 $ and $ p > 2 $ we also have $ \|g_{s,\delta}^\ast(f)\|_p \leq C\, \|f\|_p.$
		
	\end{thm}
	The inequality of $ g_{s,\delta}^\ast(f) $ follows from the fact that the maximal function 
	$$ Mf(x) = \sup_{t>0}  \, t^{-\frac{d_\K}{2\delta}}\int_{\R^d} |f(y)|  \tau_\K(x)(1+t^{-\frac{1}{\delta}} |\cdot|^2)^{-s}(-y) \, h_\kappa^2(y) dy $$
	is bounded on $ L^p , 1 < p < \infty $ provided $ s > d_\K/2.$ It is known that Dunkl translation is bounded on $ L^1 $ when restricted to radial functions. In view of this,  the above maximal function  is dominated by the Hardy-Littlewood maximal function associated to the Dunkl convolution. The details are the same as in the Fourier transform case.\\
	
	\subsection{Radial multipliers for the Dunkl transform} In this subsection we consider Dunkl multipliers $ \mathcal{T}_m $ when  $ m $ is radial. As in the case of the Fourier transform, let us define
	$$ \widehat{m}_{s,\delta}(x,t) = \int_{\R^d} E_\K(ix, \xi) \, |\xi|^{2s} m(t^{-\frac{1}{2\delta}}\xi) e^{-|\xi|^{2\delta}} \, h^2_\kappa(x) d\xi$$ 
	and consider the modified H\"ormander condition 
	\begin{equation}\label{modi-H}
		\sup_{t>0}\,\int_{\R^d} (1+|x|^2)^s |\widehat{m}_{s,\delta}(x,t)|^2\, h^2_\kappa(x) dx \leq C_s.
	\end{equation}
	As before we can rewrite this condition in terms of another function defined by
	$$ \widetilde{m}_{s,\delta}(x,t) = \int_{\R^d} E_\K(ix, \xi) \, |\xi|^{2s} m(\xi) e^{-t|\xi|^{2\delta}} \, h^2_\kappa(x) d\xi.$$
	The modified H\"ormander's condition is equivalent to the following estimate:
	\begin{equation}\label{mod-hm}
		\int_{\R^d} (1+t^{-1/\delta}|x|^2)^s |\widetilde{m}_{s,\delta}(x,t)|^2\, h^2_\kappa(x) dx \leq C t^{-2s/\delta-d_\K/(2\delta)}.
	\end{equation}
	
	\begin{thm}\label{m-rad}  Let $ m $ be a radial function which satisfies the condition \eqref{modi-H} for some $ s >d_\K/2.$ 
		Then the Dunkl multiplier $ \mathcal{T}_m $ is bounded on $ L^p(\R^d,h_\K^2 )$ for all $ 1 < p < \infty.$
	\end{thm}
	\begin{proof} Since the norm of $ g_{k+1,\delta}(\mathcal{T}_mf) $ is equivalent to that of $ \mathcal{T}_mf $ it is enough to prove that $ g_{k+1,\delta}(\mathcal{T}_mf,x) \leq C\, g_{s,\delta}^\ast(f,x).$ To this end we write 
		
		$$ G_{k+1,\delta}\mathcal{T}_mf(x,2t) = \int_{\R^d} \partial_t T_{t,\delta}f(y)\, \tau_\K(x) \widetilde{m}_{s,\delta}(-y,t)\, h_\kappa^2(y) dy.$$
		Now we make use of the assumption that $ m $ is radial. Writing  
		$$  \widetilde{m}_{s,\delta}(y,t) = (1+t^{-1/\delta}|y|^2)^{-s/2} (1+t^{-1/\delta}|y|^2)^{s/2}\widetilde{m}_{s,\delta}(y,t)$$
		and making use of Corollary \ref{trans-prod}  we get the estimate
		$$ \left( \tau_\K(x)\widetilde{m}_{s,\delta}(y,t)\right)^2 \leq \tau_\K(x)(1+t^{-1/\delta}|y|^2)^{-s}  \tau_\K(x) \left((1+t^{-1/\delta}|y|^2)^{s}|\widetilde{m}_{s,\delta}(y,t)|^2\right).$$
		Under the assumption on the function $ \widetilde{m}_{s,\delta}(x,t) $ and $L^2$-boundedness of the Dunkl translations, we get the estimate
		$$ |G_{k+1,\delta}\mathcal{T}_mf(x,2t)|^2 \leq C t^{-2s/\delta-d_\K/(2\delta)}  \int_{\R^d} |\partial_t T_{t,\delta}f(y)|^2\,\tau_\K(x)(1+t^{-1/\delta}|-y|^2)^{-s} h^2_\kappa(y) dy \,  .$$
		Integrating the above with respect to $ t^{2k+1} dt $ and noting that $ s/\delta =k $ we obtain the estimate $ g_{k+1,\delta}(T_mf,x) \leq C\, g_{s,\delta}^\ast(f,x)$ as claimed.
	\end{proof}
	
	\subsection{Dunkl multipliers acting on radial functions}  In this subsection we  consider the  general multiplier $ \mathcal{T}_m $  (where $ m $ is not assumed to be radial) acting on the subspace $ L^p_{rad}(\R^d, h_\K^2) $. A slight variation in the proof of Theorem \ref{m-rad} yields the following result.  Note that the following theorem is proved only for $ p \geq 2.$ This is due the fact that  $ g_{s,\delta}^\ast $ function is bounded on $ L^p(\R^d, h_\K^2 ) $ only for $ p > 2 $ and as $ \mathcal{T}_m $ need not take radial functions into radial functions, we cannot use duality.
	
	\begin{thm}\label{m-no-rad}  Let $ m $ be a bounded function which satisfies the condition \eqref{modi-H} for some $ s >d_\K/2.$ 
		Then the Dunkl multiplier $ \mathcal{T}_m :  L^p_{rad}(\R^d, h_\K^2 ) \rightarrow  L^p(\R^d,h_\K^2 )$ is bounded for all $ 2 \leq p < \infty.$
	\end{thm}
	\begin{proof} We begin as in the case of Theorem \ref{m-rad} the integral representation
		\begin{equation}\label{one} G_{k+1,\delta}\mathcal{T}_mf(x,2t) = \int_{\R^d} \tau_\K(x)\left(\partial_t T_{t,\delta}f\right)(-y)\,  \widetilde{m}_{s,\delta}(y,t)\, h_\kappa^2(y) dy 
		\end{equation}
		where the translation is now placed on the first factor. As before writing
		$$  \widetilde{m}_{s,\delta}(y,t) = (1+t^{-1/\delta}|y|^2)^{-s/2} (1+t^{-1/\delta}|y|^2)^{s/2}\widetilde{m}_{s,\delta}(y,t)$$ 
		and applying Cauchy-Schwarz in \eqref{one} and making use of \eqref{mod-hm} we obtain
		$$ |G_{k+1,\delta}\mathcal{T}_mf(x,2t)|^2 \leq C t^{-2s/\delta-N/(2\delta)}  \int_{\R^d} |\tau_\K(x) \left(\partial_t T_{t,\delta}f\right)(-y)|^2\,(1+t^{-1/\delta}|y|^2)^{-s} h_\kappa^2(y) dy \,  .$$
		As $ f $ is radial, so is $\partial_t T_{t,\delta}f $ and hence by Corollary \ref{trans-prod} we have
		$$ |\tau_\K(x)\left(\partial_t T_{t,\delta}f\right)(y)|^2 \leq \tau_\K(x) |\partial_t T_{t,\delta}f|^2(y) .$$
		Using this in the above inequality, transferring the translation to the second factor, namely  $(1+t^{-1/\delta}|y|^2)^{-s},$ integrating against $ t^{2k+1} dt $ and noting that $ s/\delta =k $ we obtain the  estimate 
		$$ g_{k+1,\delta}(\mathcal{T}_mf,x) \leq C\, g_{s,\delta}^\ast(f,x)$$ as before.
	\end{proof}
	
		Let us examine the condition \eqref{mod-hm} in detail. When $ t^{-1/\delta}|x|^2 \leq 1 $, this condition can be deduced from the inequality
		$$ \int_{\R^d}  |\widetilde{m}_{s,\delta}(x,t)|^2\, h^2_\kappa(x) dx \leq C t^{-2s/\delta-d_\K/(2\delta)}.$$
		Recalling the definition of  $\widetilde{m}_{s,\delta}(x,t)$, the boundedness of $ m $ together with the Plancherel theorem for the Dunkl transform implies that the above estimate holds. On the other hand when $ t^{-1/\delta}|x|^2 \geq 1 $, the estimate \eqref{mod-hm} follows from the fact that
	$$ \int_{\R^d} |x|^{2s}\, |\widetilde{m}_{s,\delta}(x,t)|^2\, h^2_\kappa(x) dx \leq C t^{-s/\delta-d_\K/(2\delta)}.$$
	When $ s =k $ is an integer so that $ \delta = 1$ the above estimate is equivalent to 
	\begin{equation}\label{esti-one}
		\int_{\R^d} \, | x^\alpha\, \widetilde{m}_{s,\delta}(x,t)|^2\, h^2_\kappa(x) dx \leq C t^{-k-d_\K/2} 
	\end{equation}
	for all $ |\alpha|= k.$ From the definition of  $ E_\K(x, y) $ it follows that
	$$ (-i)^\alpha\, x^\alpha\,  \widetilde{m}_{s,\delta}(x,t) = \int_{\R^d} \mathcal{D}^\alpha E_\K(ix, \xi) \, |\xi|^{2k} m(\xi) e^{-t|\xi|^{2}} \, h^2_\kappa(\xi) d\xi.$$
	Integrating by parts and using Plancherel we see that \eqref{esti-one} is equivalent to 
	$$ \int_{\R^d} \, |\mathcal{D}^\alpha \big( |\xi|^{2k} m(\xi) e^{-t|\xi|^{2}} \big)|^2\, h^2_\kappa(\xi) d\xi \leq C\, t^{-k-d_\K/2} $$
	At this point we need the Leibniz formula stated and proved in Proposition \ref{Leibniz}.\\

	\begin{proof}[\bf Proof of Corollary \ref{cor}] Let $ h_t(r) =  r^{2k} e^{-tr^2} $ and take $ g(\xi) = h_t(|\xi|).$ We are left with proving the estimate
		$$ \int_{\R^d} \, |\mathcal{D}^\alpha (m\,g)(\xi)|^2\, h^2_\kappa(\xi) d\xi \leq C\, t^{-k-d_\K/2} $$
		for all $ |\alpha| =k.$  From the Leibniz formula for $ \mathcal{D}(m\,g) $ proved in Proposition \ref{Leibniz} we need to verify 
		\begin{equation}\label{integ} \int_{\R^d} \, |\xi^{\beta}\, \mathcal{D}^{\gamma} m_\ell(\xi)\, D_r^j h_t(|\xi|)|^2\, h^2_\K(\xi)\, d\xi \leq C\, t^{-k-d_\K/2} 
		\end{equation}
		where $ 2j+|\gamma|-|\beta| =k.$ As $ m_\ell $ satisfies the same estimates as $ m $ we have
		$$ |\xi^{\beta}\, \mathcal{D}^{\gamma} m_\ell(\xi)| \leq C\, |\xi|^{-|\gamma|+|\beta|}.$$
		On the other hand recalling the definition of $ D_r $ it is easy to check that
		\begin{equation}\label{esti-heat} |D_r^j h_t(|\xi|)| \leq C\, |\xi|^{2k-2j}\, e^{-at |\xi|^2} 
		\end{equation}
		for some $ a >0.$ Consequently, as $ 2j+|\gamma|-|\beta| =k$ see that \eqref{integ} is bounded by
		$$ C  \int_{\R^d} \,|\xi|^{2k}\, e^{-2at |\xi|^2}h^2_\K(\xi)\, d\xi \leq C\, t^{-k-d_\K/2} $$
		Thus the modified H\"ormander condition  \eqref{mod-hm} is satisfied which completes the proof.
	\end{proof}
	
	\begin{proof}[\bf Proof of Theorem \ref{rad-mult}] As in the proof of the previous corollary, using the Leibniz formula proved in Corollary \ref{rad-Leibnitz} we need to prove the estimates
		\begin{equation}\label{integral} \int_{\R^d} \, | P_j(\xi)\, D_r^n m_0(|\xi|)\, D_r^\ell h_t(|\xi|)|^2\, h^2_\K(\xi)\, d\xi \leq C\, t^{-k-d_\K/2} 
		\end{equation}
		for all $ n+\ell \leq k, 2(n+\ell)-j = k $ for any homogeneous polynomial $ P_j $ of degree $j.$ The hypothesis on $ m_0 $ gives the estimate $ |D_r^n m_0(|\xi|)| \leq c_n\, |\xi|^{-2n}.$ Using this and the estimate \eqref{esti-heat} 
		and remembering the relation $ 2(n+\ell)-j = k,$ the integral \eqref{integral} is bounded by 
		$$C  \int_{\R^d} \,|\xi|^{2k}\, e^{-2at |\xi|^2}h^2_\K(\xi)\, d\xi \leq C\, t^{-k-d_\K/2} .$$
		This completes the proof of the theorem as the modified H\"ormander condition  is satisfied.
	\end{proof}
	
	\section*{Acknowledgements}
	This work was carried out when ST was visiting Harish-Chandra Research Institute as Infosys Visiting Professor, with SM also visiting during that period. Both authors thank HRI for the facilities and warm hospitality. The authors also thank the reviewer for her/his thorough review of the manuscript and for providing many helpful comments and suggestions. ST was partially supported by INSA, and SM was supported by research fellowship from the Department of Atomic Energy (DAE), Government of India.

	\providecommand{\bysame}{\leavevmode\hbox to3em{\hrulefill}\thinspace}
	\providecommand{\MR}{\relax\ifhmode\unskip\space\fi MR }
	\providecommand{\MRhref}[2]{%
		\href{http://www.ams.org/mathscinet-getitem?mr=#1}{#2}
	}
	\providecommand{\href}[2]{#2}
	
\end{document}